\documentclass[12pt]{article}
\usepackage{amssymb,amsmath,amsthm,hyperref,latexsym}
\usepackage{graphicx,color}
\usepackage[all]{xy}
\numberwithin{equation}{section}

\theoremstyle{plain}
\newtheorem{Theorem}{Theorem}[section]
\newtheorem{Corollary}[Theorem]{Corollary}
\newtheorem{Lemma}[Theorem]{Lemma}
\newtheorem{Proposition}[Theorem]{Proposition}

\theoremstyle{definition}

\newtheorem{Example}[Theorem]{Example}
\theoremstyle{Remark}
\newtheorem{Remark}[Theorem]{Remark}

\def\ZZ{{\bf Z}}

\def\RR{{\mathbb R}}
\def\CC{{\mathbb C}}

\def\sym#1#2{\mbox{\rm Sym}_{#1}(#2)}

\def\O{{\mathcal O}}

\def\sym{{\mathcal S}}

\def\cl#1{{\mathcal #1}}

\def\VaVa{{\cl V}\kern-5pt {\cl V}}

\def\PP{{\mathbb P}}

\def\C{{{\mathbb C}}}
\def\F{{{\mathbb F}}}
\def\sym{{{\mathrm Sym}}}
\def\g{{{\mathfrak g}}}

\begin{document}
\title{The critical space for orthogonally invariant varieties}
\author{Giorgio Ottaviani\footnote{The author is member of Italian GNSAGA-INDAM. Partially supported by the
H2020-MSCA-ITN-2018 project POEMA.}}

%\subjclass[2000]{14N07, 14N05, 14N10, 15A69, 15A18}
%\keywords{eigenvectors, tensors, singular tuples, vector bundles, Chern classes}

\maketitle

\begin{small}\emph{  Dedicated to Bernd Sturmfels on the occasion of his
60-th birthday.}\end{small}

\begin{abstract}
Let $q$ be a nondegenerate quadratic form on $V$. Let 
$X\subset V$ be invariant for the action of a Lie group
$G$ contained in $SO(V,q)$.
For any $f\in V$ consider the function $d_f$ from $X$ to $C$ defined by $d_f(x)=q(f-x)$.
We show that the critical points of $d_f$ lie in the subspace orthogonal to ${\mathfrak g}\cdot f$, that we call critical space.
In particular any closest point to $f$ in $X$ lie in the critical space.
This construction applies to singular t-ples for tensors and to flag varieties and generalizes a previous result of Draisma, Tocino and the author. As an application, we compute the Euclidean Distance degree of a complete flag variety.
\end{abstract}

\section{Introduction and main result}

Let $V$ be a complex vector space equipped with a nondegenerate symmetric bilinear form $q\in\sym^2 V$,
identified in this paper with its associated quadratic form.
The orthogonal group $SO(V,q)=SO(V)$ consists of linear transformations of $V$ leaving $q$ invariant.
 Let $X\subset V$ be an algebraic variety defined over $\RR$, this includes the case when $X$ is the cone
 over a projective variety defined over $\RR$.
 We assume that $X$ is $G$-invariant for the action of a Lie group $G\subset SO(V)$.
In many cases of interest $X$ is $H$-invariant for a larger group $H$ and we can take $G=SO(V)\cap H$,
see \S 2 for the case of partially symmetric tensors.

We denote by $\g=T_eG$ the Lie algebra of $G$, where $e$ is the identity element, note that $\g\subset\mathfrak{so}(V) $. The tangent space to the orbit $G\cdot f$ at $f$ is $f+\g\cdot f$.
Denoting by $G_f=\{g\in G| g\cdot f=f\}$  the isotropy group of $f$, we have $\dim \g\cdot f=\dim \g-\dim G_f$.

We define for $f\in V$ the {\it critical space} of $f$ as the subspace

\begin{equation}\label{eq:defcrit}H_f:=\left(\g\cdot f\right)^\perp=\left\{v\in V|q(v,w)=0\quad\forall w\in\g\cdot f\right\}
\end{equation}

We remark that $\mathrm{codim\ }H_f=\dim \g\cdot f$, so we have  $\mathrm{codim\ }H_f\le\dim\g$
and the equality holds for general $f$ in many cases, but it cannot hold in the cases when $\dim\g\ge\dim V$ (this happens
when $V=\CC^a\otimes\CC^b\otimes\CC^c$ and $c$ is large, in the setting of \S 2.2). Consider the function $d_f\colon X\to\C$,
$d_f(x)=q(f-x,f-x)$, which, in the case $f$ is real,  extends the squared distance function from $f$ defined over $\RR$.

Note that at a critical point $x$ of $d_f$ we have $f-x\in\left(T_xX\right)^\perp$.

\begin{Lemma}\label{lem:gbasic}
If $g\in\g$ then $q(g\cdot x,y)=-q(x,g\cdot y)$, in particular $q(g\cdot x, x)=0$.
\end{Lemma}
\begin{proof} Let $g(t)\subset G$ be a path such that $g(0)=e$ and $\dot{g}(0)=g$.
Taking the derivative at $t=0$ of the constant function $q(g(t)\cdot x,g(t)\cdot y)$ the thesis follows.
\end{proof}

Our main result is the following Theorem. Its proof is quite simple, nevertheless we will see in the rest of the paper
it has some nontrivial consequences.

\begin{Theorem}\label{thm:main}
Let $X$ be $G$-invariant for the action of $G\subset SO(V)$.

\begin{enumerate}

\item{} The critical points of $d_f$ on $X$ lie in $H_f$. 

\item{} When $f$ is real, any closest point to $f$ in $X_{\RR}$ (with respect to $q$) belongs to $H_f$.

\item{} $f\in H_f$.
\end{enumerate}
\end{Theorem}
\begin{proof}
Let $x$ be a critical point.
We need to prove $q(x-f,g\cdot f)=0$ $\forall g\in\g$.
We have
$q(g\cdot f,f)=0\quad\forall g\in\g$ from Lemma \ref{lem:gbasic}. So it is enough to show that 
$q(x,g\cdot f)=0$ $\forall g\in\g$.
The crucial remark is that since $X$ is $G$-invariant then $\g\cdot x\subset T_xX$.
Since $x$ is critical it follows the chain of equalities (the second and the third one by Lemma \ref{lem:gbasic}
$0=q(g\cdot x, x-f)=-q(g\cdot x,f)=q(x,g\cdot f)$, which proves (1).

(2) is an immediate consequence of (1).

(3) follows by $q(f,g\cdot f)=0$.

\end{proof}

A partial converse to Theorem \ref{thm:main} is the following. 
\begin{Theorem}\label{thm:mainconverse}
Let $X$ be $G$-invariant for the action of $G\subset SO(V)$. Let $x\in H_f\cap X$.
\begin{enumerate}
\item{}
  If the orbit $G\cdot x$ is dense in $X$ then $x$ is a critical point of $d_f$ restricted to $X$.
\item{} If $X$ is a cone, $x$ is not isotropic  and the orbit $G\cdot [x]$ is dense in $\PP X$ then 
there is $\lambda\in\C$ such that $\lambda x$ is a critical point of $d_f$ restricted to $X$.
\end{enumerate}
\end{Theorem}

\begin{proof} We have the equality $\g\cdot x = T_xX$ by assumption, and with this equality
all the steps of the proof of Theorem \ref{thm:main} are invertible. This proves (1).
The assumption of (2) implies that $\g\cdot x +\langle x\rangle= T_xX$.
Since $x$ is not isotropic there is $\lambda$ such that 
since $q(\lambda x,\lambda x-f)=0$, namely $\lambda=\frac{q(x,f)}{q(x,x)}$, so that orthogonality is guaranteed on the subspace $\langle x\rangle\subset T_xX$.
To check orthogonality on the remaining part of $T_xX$ we may replace $x$ with $\frac{q(x,f)}{q(x,x)} x$ and the same argument in (1) works.
\end{proof}
A stronger converse form will be proved for tensors, see Theorems \ref{thm:conversetensors} and
\ref{thm:conversetensors2} and for Grassmann varieties, see Theorem \ref{thm:conversetensorsgras}.
In Theorem \ref{thm:flagcomplete} we will compute the Euclidean Distance degree ($\mathrm{EDdegree}$) of a complete flag variety
with respect to the Frobenius product.

We recall that $\mathrm{EDdegree}(X)$ (introduced in \cite{DHOST} by following an idea by Bernd Sturmfels) is the number of critical points of $d_f$ restricted to $X$ for general $f$.
In many cases of interest it happens that $H_f\cap X$ is finite and reduced for general $f$, in these cases its cardinality
counts  $\mathrm{EDdegree}(X)$, see eq. (\ref{eq:binary}) and Theorem \ref{thm:flagcomplete} .

The critical space was introduced for tensors in \cite{OP} and for partially symmetric tensors in \cite{DOT}.
In Corollary \ref{cor:dot} we get an alternative proof of the fact proved in \cite{DOT}
by Draisma, Tocino and the author that any best rank $q$ approximation of a partially symmetric tensor $f$ lies in the critical space.

Our approach is somehow dual to the one in \cite{DLOT, BD}, where $\mathrm{EDdegree}$ was considered 
in an orthogonally invariant setting, but certain subvarieties of $X$ were constructed
in order to cut transversally the orbits.
%%%%%%%%%%%%%%% 
\section{Symmetric and partially symmetric tensors}
\subsection{Symmetric tensors}\label{subsec:st}
Let $W$ a space of dimension $n+1$ and $V=\sym^d W$. We assume that $W$ is equipped with a nondegenerate quadratic form $q_W$ and we choose coordinates in $W$ such that $q_W=\sum_{i=0}^nx_i^2$.
There is a unique nondegenerate bilinear form $q$ such that 
\begin{equation}\label{eq:frob}q(x^d,y^d)=q_W(x,y)^d\quad\forall x, y\in W,\end{equation} which is called the Frobenius  (or Bombieri-Weyl) form. 
Since every polynomial in $\sym^dW$ can be written as a sum of powers of linear forms,
it is enough to ask (\ref{eq:frob}) for any power $x^d$, $y^d$.
The group 
$G=SO(W,q_W)$ acts over $V$ by the analogous rule $g\cdot (x^d)=(g\cdot x)^d$. We get the inclusion $G\subset SO(V,q)$, so that 
we are in the setting of \S 1; our aim is to apply Theorem \ref{thm:main}.
The Frobenius form has the coordinate expression
$q\left(\sum_{\alpha}{d\choose \alpha}f_\alpha x^\alpha, \sum_{\alpha}{d\choose \alpha}g_\alpha x^\alpha\right) =
\sum_{\alpha}{d\choose \alpha}f_\alpha g_\alpha$
which, up to a scalar factor, has the nice M2 \cite{GS} implementation
\begin{verbatim}
diff(f,g)
\end{verbatim}
Note that $SL(W)\cap SO(V)=SO(W)$, but we will not need this fact. The monomials are orthogonal but not orthonormal 
with respect to $q$.
\begin{Proposition}\label{prop:span}
\begin{equation}\label{eq:sodij}\mathfrak{so}(W)\cdot f=\langle \frac{\partial f}{\partial x_i}x_j-\frac{\partial f}{\partial x_j}x_i\rangle_{0\le i< j\le n}\end{equation}
\end{Proposition}
\begin{proof}
It is convenient to denote
\begin{equation}\label{eq:dij}D_{ij}(f)= \frac{\partial f}{\partial x_i}x_j-\frac{\partial f}{\partial x_j}x_i
\end{equation}

For any skew-symmetric matrix $A$ we have that $e^A$ is orthogonal.
Then $f(e^{tA}x)$ is a path in the $SO$-orbit of $f$.
By taking the derivative at $t=0$ we get
$\sum_{p=0}^n\frac{\partial f}{\partial x_p}\left(Ax\right)_p\in \mathfrak{so}(W)\cdot f$
By taking $A=e_{ij}-e_{ji}$ we get exactly $D_{ij}f$ and these elements span $\mathfrak{so}(W)\cdot{f}$.
\end{proof}

The rank one tensors in $\sym^dW$ have the form $x^d$ and make a cone over the Veronese variety  $v_d\PP W$, where the origin
has been removed from the cone. We recall that the eigenvectors
of $f\in\sym^dW$ are the critical points of the function $d_f(x^d)=q(f-x^d)$ restricted to the rank one tensors \cite{Q, QZ, L}. 
In this paper we are interested in the condition $x^d\in H_f$, which does not distinguish between
$x$ and its scalar multiples, so by abuse of notation we may shift to projective space $\PP W$
and denote by the same symbol the point $x\in\PP W$.
The eigenvectors correspond to the non isotropic $x$ (i.e. $q(x)\neq 0$) such that $\nabla f(x)=x$ in $\PP W$,
which means that any representatives of the right and the left hand side differ by a nonzero scalar multiple.
The connection with (\ref{eq:sodij}) and (\ref{eq:dij}) is that the eigenvectors of $f$ make the base locus of the linear
system $\langle D_{ij}f\rangle$.

It follows from Theorem \ref{thm:main} that the eigenvectors of $f$ lie in $H_f$ (which
is obvious from the above description since $D_{ij}f$ are the minors of the matrix $\begin{pmatrix}\nabla f\\ x\end{pmatrix}$ ) and moreover the critical points of $d_f$
on the secant varieties of $d$-Veronese variety lie in $H_f$, which is not obvious from the definition and it was proved first in \cite[Theorem 1.1]{DOT}.
We will state more precisely this claim in the more general setting of partially symmetric tensors in Corollary \ref{cor:dot}.

We give now a more precise converse to Theorem \ref{thm:main} (1) in the case when
$X$ is the cone of symmetric tensors of rank one.

\begin{Theorem}\label{thm:conversetensors}
 For general $f\in\sym^dW$,
$H_f\cap v_d\PP W$ consists exactly of the critical points of $d_f$ restricted to $v_d\PP W$,
namely of the eigenvectors of $f$.
\end{Theorem}
\begin{proof} Let $v^d\in H_f\cap v_d\PP W$. In particular $q(v^d,g\cdot f)=0$ for any $g\in\g$,
which implies that $D_{ij}f$ vanishes at $v$.
This is equivalent to the matrix
$$\begin{pmatrix}\nabla f\\ x\end{pmatrix}$$
having rank one at $v$, which is the condition that $v$ is eigenvector of $f$, if $v$ is not isotropic.
By \cite[Lemma 4.2]{DLOT} the critical points of $d_f$ for a general $f$ avoid any proper closed subset of $v_d\PP W$,
so for general $f$ it is guaranteed that no isotropic $v$ is found.

\end{proof}

\begin{Remark} Note that for even $d$, $\g\cdot(f+cq^{d/2})=\g\cdot f + [q^{d/2}]$ for any $c\in\C\setminus\{0\}$.
Conversely, if $\g\cdot f=\g\cdot h$ for general $f, h$ then we get $H_f=H_h$, so that $f$, $h$ have the same eigenvectors and Turatti proves in \cite{Tur}
(generalizing previous results from \cite{BGV, ASS})
that there exists $c\in\C$ such that $f+cq^{d/2}=h$.
\end{Remark}
%%%%%%%%%%%%%%%%%%%%%%
\subsection{Partially symmetric tensors}\label{subsec:pst}
Consider the tensor product $\sym^{d_1}V_1\otimes\ldots\otimes\sym^{d_k}V_k=V$.
We assume we have nondegenerate symmetric bilinear forms $q_i$ on $V_i$.
$V$ is equipped with the Frobenius form $q$ such that on decomposable elements
$$q(v_{1}^{d_1}\otimes\ldots\otimes v_k^{d_k}, w_1^{d_1}\otimes\ldots\otimes w_k^{d_k}) = \prod_{i=1}^k q_i(v_i,w_i)^{d_i}.$$
The decomposable elements make a cone over the Segre-Veronese variety $X\simeq \PP V_1\times\ldots\times\PP V_k$ embedded in $\PP V$ with the line bundle $\O(d_1,\ldots, d_k)$. The group $G=SO(V_1,q_1)\times\ldots\times SO(V_k,q_k)$ acts on $V$, we have again the inclusion
$G\subset SO(V,q)$
 and Theorem \ref{thm:main} applies.
Denote by $x_{i,0}\ldots x_{i,n_i}$ an orthogonal coordinate system on $V_i$. Analogously to Proposition \ref{prop:span} 
the orbit $\mathfrak{so}(V_1)\times\ldots\times\mathfrak{so}(V_k)\cdot f$
is spanned by $\frac{\partial f}{\partial x_{p,i}}x_{p,j}-\frac{\partial f}{\partial x_{p,j}}x_{p,i}$ for $0\le i< j\le n_p$, $p=1,\ldots, k$.
It follows that the critical space $H_f$ defined according to (\ref{eq:defcrit}) coincides with the one defined in \cite{DOT}. 

The critical points of $d_f(x)=q(f-x)$ restricted to the Segre-Veronese variety are the singular t-ples of $f$ \cite{L},
their number is called $\mathrm{EDdegree}$ in \cite{DHOST} and it is counted by the formula in \cite{FO}, see also \cite[\S 8]{DHOST}.

The proof of Theorem \ref{thm:conversetensors} generalizes to this setting and gives

\begin{Theorem}\label{thm:conversetensors2}
 For general $f\in\sym^{d_1}V_1\otimes\ldots\otimes\sym^{d_k}V_k=V$,
 let $X\subset\PP V$ be the Segre-Veronese variety of rank one tensors.
$H_f\cap X$ consists exactly of the singular $t$-ples of $f$. 
\end{Theorem}

Since general partially symmetric tensors $f$ have trivial isotropic groups,
in the binary case $X_{\mathbf{d}}=\PP^1\times\ldots\times\PP^1$
 embedded in $\PP(\sym^{d_1}\C^2\otimes\ldots\sym^{d_k}\C^2)$ with the line bundle
 $\O(d_1,\ldots, d_k)$ we have $G=(\C^*)^k$,  $\g=\C^k$ and  the nice coincidence $\mathrm{codim}H_f=k=\dim X_{\mathbf{d}}$. Hence the cardinality of the intersection between $H_f$ and $X_{\mathbf{d}}$ can be counted
 by Bezout Theorem and it follows an alternative proof of the formula
\begin{equation}\label{eq:binary}
\mathrm{EDdegree}(X_{\mathbf{d}})=\deg X_{\mathbf{d}}=k!d_1\ldots d_k,
\end{equation}
already known from \cite{FO}, \cite[Eq. (1.6)]{Sod}. Our approach explains that the resulting equality between 
$\mathrm{EDdegree}$ and $\deg$ of $X_{\mathbf{d}}$ is not a coincidence. We will apply again this approac to complete flag varieties in
Theorem \ref{thm:flagcomplete}.

\begin{Example}
If $\dim A=\dim B=\dim C=2$ we denote by $Q_A$, (resp. $Q_B$, $Q_C$ ) the isotropic quadric
consisting of two points on $\PP(A)$ (resp. $\PP(B)$, $\PP(C)$).

We have that $\dim(\mathfrak{so}\cdot f)<3$ if and only if $f$ belongs to one of the following six $\PP^3$ linearly embedded in $\PP( A\otimes B\otimes C)$ (each item consists of two $\PP^3$'s)
$$Q_A\otimes B\otimes C,\quad A\otimes Q_B\otimes C,\quad A\otimes B\otimes Q_C.$$
\end{Example}

The following result was proved in \cite{DOT}, joint with J. Draisma and A. Tocino. The proof given here,
as a consequence of Theorem \ref{thm:main}, is maybe simpler.

\begin{Corollary}\label{cor:dot}\cite[Theorem 1.1]{DOT} Let $X_q$ be the q-secant variety to the Segre-Veronese variety in 
$\PP\left(\sym^{d_1}V_1\otimes\ldots\otimes\sym^{d_k}V_k\right)$. Then the critical points of the distance function
from a tensor $f$ to $X_q$ lie in $H_f$. In particular any best rank $q$ approximation of $f$ (when it exists) lie in $H_f$.
\end{Corollary}
%%%%%%%%%%%%%%%%%%%%%%%%%
\section{Grassmann and Flag varieties}
\subsection{Grassmann varieties}\label{subsec:grass}

Let $V=\wedge^kW$, we consider  the Grassmann variety $Gr(k,W)$ of $k$-dimensional subspaces of $W$,
its cone is embedded in $V$. Again, a nondegenerate quadratic form $q_W$ on $W$ extends to the Frobenius form $q$ on $V$
by requiring $q(v_{1}\wedge\ldots\wedge v_k, w_1\wedge\ldots\wedge w_k)=\det\left( q_W(v_i, w_j)\right)$ (the Gram determinant).
If $v=v_1\wedge\ldots\wedge v_k\in\wedge^k W$ then the derivative $ \frac{\partial v}{\partial x_i}\in\wedge^{k-1}W$ is defined by the Leibniz formula
$$ \frac{\partial v}{\partial x_i} = \sum_{j=1}^k v_1\wedge\ldots \wedge  \frac{\partial v_j}{\partial x_i}\wedge\ldots\wedge v_d$$
and extended by linearity to all $\wedge^k W$ . 
This is compatible with the inclusion $\wedge^kW\subset W^{\otimes k}$ and the form $q$ just defined is the restriction
of the Frobenius form on $W^{\otimes k}$ of the previous section.The same formula (\ref{eq:sodij}) holds formally in case $SO(W)$ acts on $\wedge^k W$ .
\begin{equation}\label{eq:sodijskew}\mathfrak{so}(W)\cdot f=\langle \frac{\partial f}{\partial x_i}\wedge x_j-\frac{\partial f}{\partial x_j}\wedge x_i\rangle_{0\le i< j\le n}\end{equation}

The EDdegree of Grassmann varieties with respect to the Frobenius form is still unknown in general. 

For a general $f\in\wedge^kW$, we have that a non isotropic $v=v_1\wedge\ldots\wedge v_k$ is a critical point for
$d_f$ if $T(v_1\wedge\ldots\widehat{v_i}\ldots\wedge v_k)=q(v_i,-)$ $\forall i=1,\ldots k$.
Again, the proof of Theorem \ref{thm:conversetensors} generalizes to this setting and gives

\begin{Theorem}\label{thm:conversetensorsgras}
 For general $f\in\wedge^kW$,
$H_f\cap Gr(k,W)$ consists exactly of the critical points of $d_f$ restricted to the Grassmann variety $Gr(k,W)$. 
\end{Theorem}

%%%%%%%%%%%%%%%%%%%%%%%%%%%
\subsection{Flag varieties}\label{subsec:flag}

For a flag variety $X=SL(W)/P$, where $P$ is a parabolic subgroup of $SL(W)$\cite[\S 23.3]{FH}, embedded by a very ample line bundle $L$,
$H_f\cap X$ consists exactly of the critical points of $d_f$. 
The embedding space is a Schur module $S^{\alpha}W$ where the Frobenius form
is defined again by restriction of the one on $W^{\otimes k}$ and again we have $G=SO(W)$.

For complete flag varieties $\F_n$, which parametrize complete flags  $(L_1\subset\ldots \subset L_n)\subset W$
with $\dim L_i=i$ (partial flags may miss some $L_i$'s), the above principle becomes effective in computing the number of critical points. We recall that $\dim \F_n=n(n+1)/2$ and that $\F_n=SL(n+1)/B$ where $B$ is the Borel subgroup of upper triangular matrices.
The following two Lemmas are well known, we include the proofs for the convenience of the reader. 
\begin{Lemma}$\chi(\F_n,\ZZ)=(n+1)!$
\end{Lemma}
\begin{proof}
A general section of the tangent bundle $T\F_n$ is given by a matrix $A\in SL(n+1)=SL(W)$
with distinct eigenvalues and corresponding eigenvectors $v_1,\ldots v_{n+1}$.
The zero locus of this section consists of $A$-invariant complete flags $(L_1\subset\ldots \subset L_n)$
with $\dim L_i=i$.
There are $(n+1)$ choices for $L_n$, obtained by the span of $n$ among the $v_i$.
For each $L_n$ there are correspondingly $n$ choices for $L_{n-1}$, and so on there are $(n+1)!$
choices for each $A$-invariant complete flag. The thesis follows from Gauss-Bonnet Theorem.
\end{proof}
\begin{Lemma}
\begin{itemize}
\item{(i)} Let $\F_n$ be embedded with the line bundle $\O(a_1,\ldots, a_n)$ in 
the projective space over $S^{a_1,\ldots, a_n}\C^{n+1}$, the module with Young diagram having
$\sum_{i=j}^{n}a_i$ boxes in the $j$-th row.
The degree of the embedded variety is
\begin{equation}\label{eq:degreeflag}{{n+1}\choose 2}!\prod_{1\le i<j\le n+1}\frac{a_i+\ldots+a_{j-1}}{j-i}.\end{equation}
\item{(ii)} When $a_i=1$ we get $\deg\F_n={{n+1}\choose 2}!$
\end{itemize}
\end{Lemma}
\begin{proof}
We have $H^0(\F_n,\O(a_1,\ldots, a_n))=\prod_{1\le i<j\le n+1}\frac{a_i+\ldots+a_{j-1}+j-i}{j-i}$
by Weyl character formula (see \cite[eq. (15.17)]{FH}).
Then the Hilbert polynomial is 
$H^0(\F_n,\O(ta_1,\ldots, ta_n))=\prod_{1\le i<j\le n+1}\frac{t(a_i+\ldots+a_{j-1})+j-i}{j-i}$
and computing the leading term we get the thesis.
In case (ii) the Hilbert polynomial simplifies to
$\chi(\F_n,\O(t,\ldots, t))=(t+1)^{{n+1}\choose 2}$.

\end{proof}

\begin{Theorem}\label{thm:flagcomplete}
Let $B\subset SL(n+1)$ be the Borel subgroup of upper triangular matrices.
For a complete flag variety $\F_n=SL(n+1)/B$ , embedded by a very ample
line bundle $L=\O(a_1,\ldots, a_n)$, with respect to the Frobenius form we have that
$\mathrm{EDdegree}\F_n=\deg \F_n$ is given by
(\ref{eq:degreeflag}). 
\end{Theorem}

\begin{proof} For general
$f\in H^0(SL(n+1)/B)$, we have again the nice coincidence that the codimension of $H_f$ is equal to the dimension of $\F_n$, 
which is ${{n+1}\choose 2}=\dim SO(n+1)$, so that 
the critical points are cut by a linear space of  complementary dimension.
\end{proof}

\begin{Example} For $n=2$, the flag variety $SL(3)/B$ embedded with $\O(a,b)$ has
$\mathrm{EDdegree}\F_3=\deg \F_3=3ab(a+b)$.
\end{Example}

%%%%%%%%%%%%%%%%%%%%%%%%%%%%

\addcontentsline{toc}{chapter}{Bibliography} \noindent
\textit {Dipartimento di Matematica e Informatica ``Ulisse Dini'',  University of Florence, viale Morgagni 67/A, I-50134, Florence, Italy\\
\vskip 0.1cm
\noindent giorgio.ottaviani@unifi.it}
\end{document}